\documentclass[11pt,twoside]{amsart}
\usepackage{bbm}
\usepackage{bbm}
\usepackage{mathrsfs}
\usepackage{amsmath}
\usepackage{amsthm}
\usepackage{young}
\usepackage[enableskew]{youngtab}
\usepackage{amsfonts}
\usepackage{amssymb}
\usepackage{latexsym}
\usepackage[all]{xy}
\usepackage{mathrsfs}
\usepackage{}
\usepackage{bbding}
\usepackage{amssymb}
\usepackage{amsmath}
\usepackage[enableskew]{youngtab}
\usepackage{}
\usepackage{mathrsfs}
\usepackage{mathrsfs}
\usepackage{mathrsfs}
\usepackage{mathrsfs}
\usepackage{mathrsfs}
\usepackage{mathrsfs}
\usepackage{amsfonts}
\usepackage{mathrsfs}
\usepackage{mathrsfs}
\usepackage{mathrsfs}
\usepackage{mathrsfs}
\usepackage{mathrsfs}
\usepackage{mathrsfs}
\usepackage{amsfonts}
\usepackage{amsfonts}
\usepackage{amsfonts}
\usepackage{mathrsfs}
\usepackage{amsfonts}
\usepackage{amsfonts}
\usepackage{amsfonts}
\usepackage{amsfonts}
\usepackage{amsfonts}
\usepackage{amsfonts}
\usepackage{amsfonts}
\usepackage{mathrsfs}
\usepackage{mathrsfs}
\usepackage{amsfonts}
\usepackage{amssymb}
\usepackage{mathrsfs}
\usepackage{amsfonts}
\usepackage{mathrsfs}
\usepackage{amsfonts}
\usepackage{amsfonts}
\usepackage{mathrsfs}
\usepackage{mathrsfs}
\usepackage{mathrsfs}
\usepackage{amsfonts}
\usepackage{amsfonts}
\usepackage{amsfonts}
\usepackage{amsfonts}
\usepackage{amsfonts}
\usepackage{amsfonts}
\usepackage{amsfonts}
\usepackage{amsfonts}
\usepackage{amsfonts}
\usepackage{amsfonts}
\usepackage{mathrsfs}
\usepackage{amsfonts}
\usepackage{amsfonts}
\usepackage{amsfonts}
\usepackage{amsfonts}
\usepackage{amsfonts}
\usepackage{amsfonts}
\usepackage{amsfonts}
\usepackage{amsfonts}
\usepackage{mathrsfs}
\usepackage{mathrsfs}
\usepackage{mathrsfs}
\usepackage{mathrsfs}
\usepackage{mathrsfs}
\usepackage{mathrsfs}
\usepackage{amsfonts}
\usepackage{amssymb}
\usepackage{mathrsfs}
\usepackage{mathrsfs}
\usepackage{mathrsfs}
\usepackage{amsfonts}
\usepackage{amsfonts}
\usepackage{amsfonts}
\usepackage{mathrsfs}
\usepackage{amsfonts}
\usepackage{amsfonts}
\usepackage{mathrsfs}
\usepackage{amsfonts}
\usepackage{mathrsfs}
\usepackage{amsfonts}
\usepackage{amsfonts}
\usepackage{mathrsfs}
\usepackage{stmaryrd}
\usepackage{stmaryrd}
\usepackage{stmaryrd}
\usepackage{mathrsfs}
\usepackage{amsfonts}
\usepackage{amsfonts}
\usepackage{amsfonts}
\usepackage{amsfonts}
\usepackage{amsfonts}
\usepackage{amsfonts}
\usepackage{amsfonts}
\usepackage{amsfonts}
\usepackage{amsfonts}
\usepackage{amsfonts}
\usepackage{bm}
\usepackage{young}
\usepackage[enableskew]{youngtab}
\usepackage{tipa}
\usepackage{tipa}
\usepackage{amsfonts}
\usepackage{amsfonts}
\usepackage{amsfonts}
\usepackage{amsfonts}
\usepackage{amsfonts}
\usepackage{amsfonts}
\usepackage{amsfonts}
\usepackage{amsfonts}
\usepackage{amsfonts}
\usepackage{amsfonts}
\usepackage{amsfonts}
\usepackage{amsfonts}
\usepackage{amsfonts}
\usepackage{amsfonts}
\usepackage{amsfonts}
\usepackage{amsfonts}
\usepackage{amsfonts}
\usepackage{amsfonts}
\usepackage{amsfonts}
\usepackage{amsfonts}
\usepackage{amsfonts}
\usepackage{amsfonts}
\usepackage{amsfonts}
\usepackage{amsfonts}
\usepackage{amsfonts}
\usepackage{amsfonts}
\usepackage{amsfonts}
\usepackage{amsfonts}
\usepackage{amsfonts}
\usepackage{amsfonts}
\usepackage{amsfonts}
\usepackage{amsfonts}
\usepackage{amsfonts}
\usepackage{amsfonts}
\usepackage{amsfonts}
\usepackage{amsfonts}
\usepackage{amsfonts}
\usepackage{amsfonts}
\usepackage{amsfonts}
\usepackage{amsfonts}
\usepackage{amsfonts}
\usepackage{amsfonts}
\usepackage{amsfonts}
\usepackage{mathrsfs}
\usepackage{mathrsfs}
\usepackage{amsfonts}
\usepackage{mathrsfs}
\usepackage{amssymb}
\usepackage{amsmath}
\usepackage{amsthm}
\usepackage{amsfonts}
\usepackage{pb-diagram}

\date{}
\allowdisplaybreaks[4] \footskip=15pt
\renewcommand{\uppercasenonmath}[1]{}

\makeatletter
   \DeclareMathSizes{\@xipt}{\@xipt}{7}{2}
\makeatother
\numberwithin{equation}{section} \theoremstyle{plain}
\newtheorem*{theorem*}{Main Theorem}
\newtheorem{theorem}{Theorem}[section]

\newtheorem*{corollary*}{Corollary}
\newtheorem{lemma}[theorem]{Lemma}
\newtheorem*{lemma*}{Lemma}
\newtheorem{proposition}[theorem]{Proposition}
\newtheorem*{proposition*}{Proposition}
\newtheorem{remark}[theorem]{Remark}
\newtheorem*{remark*}{Remark}
\newtheorem{example}[theorem]{Example}
\newtheorem*{example*}{Example}
\newtheorem{definition}[theorem]{Definition}
\newtheorem*{definition*}{Definition}

\newtheorem*{conjecture*}{Conjecture}
\newtheorem*{ack*}{ACKNOWLEDGEMENTS}

\newcommand{\bc}{\begin{center}}
\newcommand{\ec}{\end{center}}

\newcommand{\pf}{\noindent\begin {proof}}
\newcommand{\epf}{\end{proof}}

\begin{document}
\newcommand{\Ho}{H}\newcommand{\Tom}{T}
\newcommand{\Eo}[1]{E^{#1}}
\newcommand{\To}[1]{T^{#1}}
\begin{center}
{\large  \bf Boltje-Maisch Resolutions of Specht Modules }

\vspace{0.8cm} {\small \bf  Xingyu Dai$^{a,\;(1)}$, \ \  Fang Li$^{a,\;(2)}$, \ \ Kefeng Liu$^{a,b\;(3)}$}\\
\vspace{0.1cm}
$^a$Center of Mathematical Sciences, Zhejiang University, Zhejiang 310027, China\\
$^b$Department of Mathematics, University of California, Los Angeles, USA\\
E-mail: $(1)$ daixingyu12@126.com;  $(2)$ fangli@zju.edu.cn;  $(3)$ liu@math.ucla.edu

\end{center}

\bigskip
\centerline { \bf  Abstract}

In \cite{21}, Boltje and Maisch found a permutation complex of
Specht modules in representation theory of Hecke algebras, which is
the same as the Boltje-Hartmann complex appeared in the
representation theory of symmetric groups and general linear groups.
In this paper we prove the exactness of Boltje-Maisch complex in the
dominant weight case.

\leftskip10truemm \rightskip10truemm
 \noindent
\\
\vbox to 0.3cm{}

 \leftskip0truemm
\rightskip0truemm
\bigskip

\section{\bf Introduction}

In module category of the group algebra $R\mathfrak{S}_r$ over an arbitrary commutative ring $R$,
 Hartmann and Boltje constructed a finite chain complex in \cite{24} for any composition $\lambda$ of a positive integer $r$.
 Almost all factors are constructed by restricted subsets of homomorphisms between permutation modules except the last one,
 which is the dual of Specht module $S^\lambda$.
Partial exactness results of position -1 and 0 about this complex were already achieved
in Hartmann and Boltje's work \cite{24} and a full proof of the exactness was obtained recently in \cite{2} with help of Bar resolution in homology theory and the construction using the Schur functor. In \cite{28} \cite{30}
 \cite{31} \cite{15} \cite{33} \cite{3} \cite{32} \cite{34} \cite{5}  \cite{29} and etc.,
some other permutation resolutions of Specht modules have been established.

In \cite{21}, the construction of the chain complex of $R\mathfrak{S}_r$  was lifted  to a chain complex of modules of the Iwahori-Hecke algebra $\mathcal {H}_r$ with an integral domain $R$ and the original chain complex reproduced by the specialization
$q=1$ and moreover. This work and its method are useful for our discussion here.

The construction of this complex was
completely combinatorial and characteristic free. It was conjectured that this
chain complex is exact whenever $\lambda$ is a partition. Also, some partial exactness has been found in \cite{10}.
In this paper, we find a way to prove this conjecture. In order to prove its exactness in the dominant weight case, we follow the method of \cite{2}, which constructs
 a bar resolution in the Borel subalgebra case and transformed it into the module category of $q$-Schur algebras by induced Functors.

The paper is organized as follows. In section 2, we find an ideal sequence $J_0\supseteq J_1\supseteq\cdots\supseteq0$, and use $J_0$ and $J_1$ to construct a bar resolution of the module in the representations of Borel subalgebra $S^+_R(n,r)$, just as
Ana Paula Santana and Ivan Yudin did in the case of symmetric groups, as in \cite{2}. However, since the proof of vanishing theorem in \cite{5}
failed for the case of Iwahori Hecke algebra's. In section 3, we use different tools to prove the module $R_\lambda$ is $S_R(n,r)\otimes_{S^+_R(n,r)}$-acyclic,
 which was introduced in \cite{20}. After that, we reach the main results Theorem \ref{main1} and Theorem \ref{main2} which give the positive answer of exactness
 of \emph{Boltje-Maisch} given in \cite{21}.

\section{\bf Notations and Quoted Results}
\subsection{Combinatorics}\label{comb}\cite{2}

   For any natural number $m$ we
  denote by $\bar{n}$ the set $\{1,\ldots,n\}$. Given a finite
  set $X$, for each map $\mu$ from $X$ to no negative integers $\mathbb{N}_0$, we can
  define its length $|\mu|:=\sum_{x\in X}\mu_x$, where we realize the map $\mu$ as $\mu(x)=\mu_x$ for any $x\in X$.
  Then, we can write a subset $\Lambda(X;r)$ as $\{\mu:X\rightarrow\mathbb{N}_0|\  |\mu|=r\}$, and get a map from
  the set $X^r$ to $\Lambda(X;r)$ as following:
\begin{eqnarray*}
&&wt:X^r\rightarrow\Lambda(X;r) \\
\text{with}&&wt(u)_x:=\#\{s|u_s=x,s=1,\ldots,r\}.
\end{eqnarray*}

There are we specific cases of the definitions given
above. First, we denote $I(n,r)$ as the set $\bar{n}^r$. The elements of $I(n,r)$
are usually called \emph{multi-indices} in other's work as \cite{2}, which can be denoted by bold
letters such as $\textbf{i}$, where $\textbf{i}:=(i_1,\cdots,i_r)$ for $1\leq i_k\leq n$ and $1\leq k\leq r$.
Also, we can denote the set $\bar{n}\times\bar{n}$ alternatively as $\{(i,j)|1\leq i,j\leq n\}$. Similarly, we can identify the set
$(\bar{n}\times\bar{n})^r$ and $I(n,r)\times I(n,r)$ via the map
$((i_1,j_1),\ldots,(i_r,j_r))\mapsto((i_1,\ldots,i_r),(j_1,\ldots,j_r))$ without any confusions.

For convenience, from now on sets $\Lambda(\bar{n};r)$ and $\Lambda(\bar{n}\times\bar{n};r)$
are denoted by $\Lambda(n,r)$ and $\Lambda(n,n;r)$,
respectively. We can notice that the elements of $\Lambda(n,r)$ are the same as
compositions of $r$ into $n$ parts, which also has same notation as $\Lambda(n,r)$ in many works as in \cite{7}.
On the set $I(n,r)$ we defined the ordering $\leq$ by
$$\textbf{i}\leq \textbf{j}\Longleftrightarrow \  i_1\leq j_1,\  i_2\leq j_2,\ldots,\  i_r\leq j_r,$$
and write $\textbf{i}<\textbf{j}$ if
$\textbf{i}\leq \textbf{j}$ and
$\textbf{i}\neq\textbf{j}$.

\subsection{Symmetric groups and Iwahori-Hecke algebras}

A {\emph{composition}} $\lambda$ of $r$ is a finite sequence of
non-negative integers $(\lambda_1,\lambda_2,\ldots,\lambda_n)$ such
that $|\lambda|=\sum_i\lambda_i=r$. Moreover, there is a partial
order $\unlhd$(resp. $\unrhd$) within compositions of $r$ as: we
denote $\lambda\unlhd\mu$ when
$\sum_{i=1}^k\lambda_i\leq\sum_{i=1}^k\mu_i$(resp.
$\sum_{i=1}^k\lambda_i\geq\sum_{i=1}^k\mu_i$) for all $1\leq k\leq
n$.

Let $\mathfrak{S}_r$ denote the symmetric
group of all permutations of $1,\ldots,r$ with Coxeter generators
$s_i:=(i,i+1)$, and $\mathfrak{S}_\lambda$ the Young
subgroup corresponding to the composition $\lambda$ of $r$. Thus, we
have
\begin{eqnarray*}
\mathfrak{S}_\lambda=\mathfrak{S}_\textbf{a}=\mathfrak{S}_{\{1,\ldots,a_1\}}\times\mathfrak{S}_{\{a_1+1,\ldots,a_2\}}\times\cdots\times\mathfrak{S}_{\{a_{n-1}+1,\ldots,a_n\}},
\end{eqnarray*}
where $\textbf{a}=[a_0,a_1,\ldots,a_n]$ with $a_0=0$ and
$a_i=\lambda_1+\cdots+\lambda_i$ for all $i=1,\ldots,n$. We
denote by $\mathscr{D}_\lambda$ the set of distinguished
representatives of right $\mathfrak{S}_\lambda$-cosets and write
$\mathscr{D}_{\lambda\mu}:=\mathscr{D}_\lambda\cap\mathscr{D}_\mu^{-1}$,
which is the set of distinguished representatives of double cosets
$\mathfrak{S}_\lambda\setminus\mathfrak{S}_r/\mathfrak{S}_\mu$.

As usual one identifies composition $\lambda$ with {\em{Young
diagram}} and we say that $\lambda$ is the \emph{shape} of the
corresponding Young diagram. For example, we can represent the
partition $(3,2)$ as $\tiny\yng(3,2)$. A $\lambda$-tableau is a
filling of the $r$ boxes of the Young diagram of $\lambda$ of the
numbers $1,2,\ldots,r$. We denote the set of $\lambda$-tableaux by
$\mathcal {T}(\lambda)$ and usually denote an element of
$\mathcal {T}(\lambda)$ as $\mathfrak{t}$.

The group $\mathfrak{S}_r$ acts from the right on $\mathcal {T}(\lambda)$ by simply applying an element $w\in \mathfrak{S}_r$ to the entries of the tableau $\mathfrak{t}\in \mathcal {T}(\lambda)$. This action is free and transitive, and it yields a bijection
\begin{eqnarray*}
\mathfrak{S}_r \xrightarrow{\sim}\mathcal {T}(\lambda), \qquad w\mapsto t^\lambda w.
\end{eqnarray*}

A $\lambda$-tableau $\mathfrak{t}$ is called \emph{row-standard} if its entries are increasing in each row form left to right. The row-standard tableaux form a subset $\mathcal{T}^{rs}(\lambda)$
of $\mathcal{T}(\lambda)$. Two $\lambda$-tableaux $\mathfrak{t}_1$ and $\mathfrak{t}_2$ are called \emph{row-equivalent} if $\mathfrak{t}_1$ and $\mathfrak{t}_2$ can arise from each other by rearranging elements within each row.
We denote the row-equivalent class of $\mathfrak{t}$ by $\{\mathfrak{t}\}$ and the set of row equivalent classes by $\overline{\mathcal {T}}(\lambda)$.
One has the canonical bijections
\begin{eqnarray}
\mathscr{D}_\lambda\xrightarrow{\sim}\mathcal{T}^{rs}(\lambda)\xrightarrow{\sim}\overline{\mathcal{T}}(\lambda)
\end{eqnarray}
given by $d\mapsto\mathfrak{t}^\lambda d$ and $\mathfrak{t}\mapsto \{\mathfrak{t}\}$.

\begin{definition}
Let $R$ be a commutative domain with 1 and let $q$ be a unitary element of $R$. The \emph{Iwahori-Hecke algebra} $\mathcal {H}_r=\mathcal{H}_{R,q}(\mathfrak{S}_r)$ of $\mathfrak{S}_r$
is the unital associative $R$-algebra with generators $T_1,T_2,\ldots,T_{r-1}$ and relations:
\begin{eqnarray*}
(T_i-q)(T_i+1)&=&0,\qquad\qquad\qquad\qquad \qquad \text{for}\  i=1,2,\ldots,r-1,\\
T_iT_j&=&T_jT_i, \qquad\qquad\qquad\qquad\  \   \text{for}\  1\leq i<j-1\leq r-2,\\
T_iT_{i+1}T_i&=&T_{i+1}T_iT_{i+1},\qquad \qquad\qquad \text{for}\  i=1,2,\ldots,r-2.
\end{eqnarray*}
$\mathcal {H}_r$ is a free $R$-module with a finite $R$-basis $\{T_w|w\in \mathfrak{S}_r\}$. $T_w$ is defined as $T_{i_1}T_{i_2}\cdots T_{i_s}$ if $w$ has a reduced presentation $w=s_{i_1}s_{i_2}\cdots s_{i_s}$. Then, we put $x_\mu:=\sum_{w\in \mathfrak{S}_\mu}T_w$ for any Young subgroup $\mathfrak{S}_\mu$, which is an element in $\mathcal{H}_r$, and define $M^\mu$ to be the right $\mathfrak{S}_r$-module $x_\mu\mathcal{H}_r$.

\end{definition}

\begin{definition}\label{double}
Fix a non-negative integer $r$, the $q$-{\em{Schur algebra}} is the endomorphism
algebra
\begin{eqnarray*}
S_R(n,r)=\text{End}_{\mathcal{H}_r}\big(\bigoplus\limits_{\mu\in \Lambda(n,r)}M^\mu\big)=\bigoplus\limits_{\lambda,\mu\in\Lambda(n,r)}\text{Hom}_{\mathcal{H}_r}(M^\mu,M^\lambda).
\end{eqnarray*}

\end{definition}

For any $\lambda, \mu \in\Lambda(n,r)$,
each element $d$ of the set
$\mathscr {D}_{\lambda\mu}:=\mathscr{D}_\lambda\cap\mathscr{D}_\mu^{-1}$ is a representative of the double cosets $\mathfrak{S}_\lambda\setminus\mathfrak{S}_r/\mathfrak{S}_\mu$.
 By \cite{10},
$\mathscr{D}_{\lambda\mu}$ forms an $R$-basis for $Hom_{\mathcal{H}_r}(M^\mu,M^\lambda)$, where $\psi^d_{\lambda\mu}$ is defined by
\begin{eqnarray}
\psi^d_{\lambda\mu}(x_\mu)=\sum\limits_{w\in \mathfrak{S}_\lambda d \mathfrak{S}_\mu} T_w=x_\lambda\sum\limits_{e\in \mathcal{D}_\nu\cap\mathfrak{S}_\mu}T_{de}=x_\lambda T_d\sum\limits_{e\in \mathcal{D}_\nu\cap\mathfrak{S}_\mu}T_e,
\end{eqnarray}
for any parameter $d\in \mathscr{D}_{\lambda\mu}$,
 and where $\nu\in \Lambda(n,r)$ is determined by Young subgroup $\mathfrak{S}_\nu := d^{-1}\mathfrak{S}_\lambda d\cap\mathfrak{S}_\mu$.

We can identify the set $\mathscr {D}_{\lambda\mu}$ with a combinatorial set $\mathcal {T}(\lambda,\mu)$,
which is call a set of generalized tableaux of $\lambda$ with content $\mu$.
Each element $T\in \mathcal {T}(\lambda,\mu)$ can be realized as
Young diagram of composition $\lambda$ whose boxes are filled
with $r$ positive integers where $\mu_1$ entries equal to 1, $\mu_2$ entries equal to 2, etc.

We denote $T^\lambda_\mu$ by the generalized tableau whose boxes are filled in the natural order,
which is the stabilizer of left group action $(wT)(i):=T(iw)$.
 If two generalized tableaux $T_1, T_2$ arised form each other by rearranging the entries within the rows,
  We call they \emph{row-equivalent}
 , i.e., $T_2=T_1 w$ for some $w\in \mathfrak{S}_\lambda$.
 We denote the set of row equivalent classes by $\overline{\mathcal{T}}(\lambda,\mu)$.
Furthermore, we can achieve a bijection $\mathfrak{S}_\lambda\setminus\mathfrak{S}_r/\mathfrak{S}_\mu\xrightarrow{\sim}\overline{\mathcal{T}}(\lambda,\mu)$, $\mathfrak{S}_\lambda w \mathfrak{S}_\mu\mapsto {wT^\lambda_\mu}$.

A generalized tableau $T\in \mathcal{T}(\lambda,\mu)$ is said to be \emph{row-semistandard} if its entries are in the natural order from left to right. Write by $\mathcal{T}^{rs}(\lambda,\mu)$ to denote the set of all such tableau. Every row-equivalent class possesses a unique row-semistandard element. It follows that $\mathcal{T}^{rs}\xrightarrow{\sim}\overline{\mathcal {T}}(\lambda,\mu), T\mapsto\{T\}$, is
a bijection. Therefore, we have reached a canonical bijection:
\begin{eqnarray}\label{bijection}
\mathscr {D}_{\lambda\mu}\xrightarrow{\sim}\mathfrak{S}_\lambda\setminus\mathfrak{S}_r/\mathfrak{S}_\mu\xrightarrow{\sim}\overline{\mathcal{T}}(\lambda,\mu)\xleftarrow{\sim}\mathcal{T}^{rs}(\lambda,\mu)
\end{eqnarray}

Moreover, {Theorem 4.7 of \cite{7} shows that
$\text{Hom}_\mathcal{H}(M^\mu,M^\lambda)$ is free as an $R$-module with
basis $\{\psi^d_{\lambda\mu}|d\in\mathscr{D}_{\lambda\mu}\}$. The
$q$-Schur algebra $S_R(n,r)$ also can be written as a free
$R$-module $\bigoplus\limits_{\substack{\lambda,\mu\in
\Lambda(n,r)\\ d\in
\mathscr{D}_{\lambda\mu}}}R\psi_{\lambda\mu}^d$.

\subsection{Quantized enveloping algebras}
In order to show the vanishing theorem in next section, we need some
notations about quantum groups here. Recall the definition of the
quantized enveloping algebra in the version given by Kashiwara.

Let $A=(a_{i,j})_{i,j\in I}$ be a
finite-type \emph{Cartan matrix}. Which means: fix a sequence of positive integers $(d_i)_i\in I$ such
that $d_ia_{i,j}=d_ja_{j,i}$ for all $i,j\in I$.
In addition, we need it satisfy the structure of root datum as in \cite{26}:

(i)  A perfect pairing $<,>:P^*\times P\rightarrow\mathbb{Z}$, where $P$ and $P^*=\text{Hom}(P,\mathbb{Z})$ are finitely generated free $\mathbb{Z}$-modules.

(ii)  Linearly independent subsets $\{\alpha_i|i\in I\}$ of $P$ and $\{\alpha_i^\vee|i\in I\}$ of $P^+$, satisfying $a_{i,j}=<\alpha_i^\vee,\alpha_j>$ for all $i$ and $j$.

$P$ and $P^*$ are called the lattices of \emph{weight} and \emph{coweights}, respectively; the $\alpha_i$ are the \emph{simple roots}, the $\alpha^\vee_i$ are the \emph{simple coroot}.
The \emph{dominance order} on $P$ is defined by $\lambda\geq\mu$ if and only if $\lambda-\mu$ can be written as a sum of simple roots. A weight $\lambda$ is \emph{dominant} (resp. \emph{antidominant}) if all $<\alpha_i^\vee,\lambda>$ are nonnegative (resp. nonpositive).

\begin{definition}
Let $U$ be the $\mathbb{Q}(q)$-algebra with generators $e_i,f_i,q^h$, $1\leq i \leq n$, $h\in P^*$, and relations
\begin{eqnarray*}
q^0&=&1, \qquad \qquad q^hq^{h'}=q^{h+h'},\\
q^he_i&=&q^{<h,\alpha_i>}e_i q^h,\\
q^hf_i&=&q^{-<h,\alpha_i>}f_iq^h,\\
\left[e_i,f_j\right]&=&\delta_{i,j}\frac{t_i-t_i^{-1}}{q_i-q_i^{-1}},\\
\sum\limits_{l=0}\limits^{a}(-1)^le_i^{(l)}e_je_i^{(a-l)}&=&\sum\limits_{l=0}\limits^{a}(-1)^lf_i^{(l)}f_jf_i^{(a-l)}=0,\\
 \text{where}\  i\neq j,\  a=1-a_{i,j}.
\end{eqnarray*}
For convenience, we have set the following abbreviations in above relations: $t_i=q^{d_i\alpha_i^{\vee}}$, $q_i=q^{d_i}$, $e^{(l)}_i=e_i^l/[l]^!_i$, $f^{(l)}_i=f_i^l/[l]^!_i$.
The subscript $i$ in $[l]^!_i$ means that the $q$ in the definition of
$[l]!$ is replaced by $q_i$.

\end{definition}

Take subset $J\subseteq I$. Set three set which consist of elements of $U$:
\begin{eqnarray*}
\mathscr{E}_J&=&\{e_i^{(s)}|s\geq 0, i\in J\}.\\
\mathscr{F}_J&=&\{f_i^{(s)}|s\geq 0, i\in J\}.\\
\mathscr{H}&=&\{q^h,\left[\begin{array}{c} q^h;0\\s  \end{array}\right]\} \qquad \text{where} \quad
\left[\begin{array}{c}x;c\\ t\end{array}\right]=\prod\limits_{s=1}\limits^t\frac{xq^{(c-s+1)}-x^{-1}q^{-(c-s+1)}}{q^s-q^{-s}}.
\end{eqnarray*}

We write $\mathscr{A}$ for the ring $\mathbb{Z}[q,q^{-1}]$ of
integral Laurent polynomials in the indeterminate $q$.
Let $U^-_\mathscr{A}(J)$, $U^+_\mathscr{A}(J)$ be the $\mathscr{A}$-subalgebras of $U$ generated respectively by $\mathscr{F}_J$, $\mathscr{E}_J$.
Similarly, we can define these subalgebra structure over any arbitrary commutative ring $R$ and denote it by $U_R^+(J)$ or $U_R^-(J)$, if there exist a
with ring homomorphism from $\mathscr{A}$ to $R$, especially $\mathbb{Q}(q)$ for example.

 For convenience, We usually omit
 the subscript when the algebras are defined over $\mathbb{Q}(q)$,
 Then, we use the following abbreviations:
 $U_\mathscr{A}^\mathfrak{b}=U_\mathscr{A}^-(I)$,
$U_\mathscr{A}^\#=U_\mathscr{A}^+(I)$, and define $U^\mathfrak{b}, U^\#$ (resp. $U_R^\mathfrak{b}, U_R^\#$) similarly with the ground ring $\mathbb{Q}(q)$ (resp. $R$).

Write $U_R$-Int for the category of all integrable
$U_R$-modules $V$. It is known as in Lusztig \cite{17} that
$U$-Int is semisimple with simple modules $\triangle(\lambda)$,
$\lambda\in P^+$, where $\triangle(\lambda)$ is the unique
maximal integrable quotient of \emph{Verma module} $M(\lambda)$.
See \cite{14}, We usually call it \emph{Weyl module} of highest weight $\lambda$.

There is a $\mathbb{Q}(q)$-algebra anti-automorphism $u\mapsto u^\tau$ of $U$ given by
\begin{eqnarray*}
e_i^\tau=f_i,\qquad f_i^\tau=e_i,\qquad (q^h)^\tau=q^h.
\end{eqnarray*}
If $V\in U_R\text{-Int}$, its \emph{contravariant dual} $V^\circ$ is the linear dual
$Hom_R(V, R)$, with its natural right $U_R$-action transferred to the left via $\tau$.
And, we denote some new objects in category $U_R$-Int as $\triangledown_R(\lambda)=\triangle_R(\lambda)^\circ$.

\section{\bf Coordinate ring and bar resolutions}

\begin{definition}\label{schur}\cite{4}
  For a commutative ring $R$, let $R[M_n(q)]$ be the associative algebra
  over $R$ generated by $X_{ij}$ with $1\leq i,j\leq n$ such that
\begin{equation}\label{sdual}
\left\{
\begin{array}{cc}
X_{ij}X_{ik}=qX_{ik}X_{ij},&if\   j>k,\\
X_{ji}X_{ki}=X_{ki}X_{ji},&if\   j>k,\\
X_{ij}X_{rs}=q^{-1}X_{rs}X_{ij},&if\   i>r,j<s,\\
X_{ij}X_{rs}-X_{rs}X_{ij}=(q^{-1}-1)X_{is}X_{rj},&if\   i<r, j<s
\end{array}
\right.
\end{equation}
As an $R$-module, $R[M_n(q)]$ has a basis
$\{\prod\limits_{ij}X^{t_{ij}}_{ij}|t_{ij}\in\mathbb{Z}^+\}$, where
the products are formed with respect to any fixed order of the
$X_{ij}$'s. Let $A_q(n,r)$ be the $r$th homogeneous component of
$R[M_n(q)]$. Then $A_q(n,r)$ has a basis
\begin{eqnarray*}
\{X_{\lambda\mu}^d := X_{\textbf{\emph{i}}_\lambda
d,\textbf{\emph{i}}_\mu}|\lambda,\mu\in\Lambda(n,r),\
d\in\mathscr{D}_{\lambda\mu}\},
\end{eqnarray*}
where $\mathscr{D}_{\lambda\mu}$ denotes the set of distinguished
representatives for $\mathfrak{S}_\lambda
\setminus \mathfrak{S}_r/ \mathfrak{S}_\mu$ (see Definition \ref{double}), and
$X_{\textbf{\emph{i}}\textbf{\emph{j}}}=X_{i_1,j_1}X_{i_2,j_2}\cdots
X_{i_r,j_r}$ if $\textbf{\emph{i}}=(i_1,\cdots, i_r)$ and
$\textbf{\emph{j}}=(j_1,\cdots,j_r)$. Denote by $A_q(n,r)^*$ the
linear dual of $A_q(n,r)$. Then, by \cite{4}
$$\varphi:\  End_{\mathcal {H}_r}(\bigoplus\limits_{\lambda\in\Lambda(n,r)}x_\lambda\mathcal{H}_r)\cong A_q(n,r)^*,$$
where the natural basis for $q$-Schur algebra $ \text{End}_{\mathcal
{H}_r}(\bigoplus\limits_{\lambda\in\Lambda(n,r)}x_\lambda\mathcal{H}_r)$
is given as follows:

 For $\lambda,\mu \in \Lambda(n,r), d\in
\mathscr{D}_{\lambda\mu}$,
if we identify $\psi_{\lambda\mu}^d$ with its images under the
isomorphism above. The basis $\{\psi_{\lambda\mu}^d\}$ is the dual
of the basis $\{X_{\lambda\mu}^d:= X_{\textbf{\emph{i}}_\lambda
d,\textbf{\emph{i}}_\mu}|\lambda,\mu\in\Lambda(n,r),\
d\in\mathscr{D}_{\lambda\mu}\}$ for $A_q(n,r)$. Moreover by
\cite{4}, we have
$\varphi(\psi^d_{\lambda\mu})(X_{\rho\nu}^{d_1})=\delta_{\lambda\rho}\delta_{\mu\nu}\delta_{d,d_1}$.
Sometimes we denote the basis
$\{\varphi(\psi_{\lambda\mu}^d)|\lambda,\mu\in \Lambda(n,r),\ d\in
\mathscr{D}_{\lambda\mu}\}\subset A_q(n,r)^*$ as
$\{\psi_{\textbf{\emph{i}}_\lambda
d,\textbf{\emph{i}}_\mu}|\lambda,\mu\in \Lambda(n,r),\ d\in
\mathscr{D}_{\lambda\mu}\}$.
\end{definition}
\begin{remark}
\textbf{Combination correspondence:}

By using the notations of section \ref{comb}, we can identify the
following three sets:
\begin{eqnarray}\label{combiso}
\Xi:\bigsqcup\limits_{\lambda,\mu\in\Lambda(n;r)}
\mathscr{D}_{\lambda\mu}&\rightarrow&I(n,r)\times I(n,r)/\sim
\quad\rightarrow \Lambda(n,n;r)\\
d\in \mathscr{D}_{\lambda\mu}&\mapsto&
\qquad\overline{(\textbf{\emph{i}}_\lambda
d,\textbf{\emph{i}}_\mu)}\qquad\mapsto\qquad wt(\textbf{\emph{i}}_\lambda
d,\textbf{\emph{i}}_\mu)
\end{eqnarray}
Where we put $I(n,r)\times I(n,r)$ as a quotient set of $I(n,r)\times I(n,r)$ with following relation
$$(\textbf{\emph{i}},\textbf{\emph{j}})\sim (\textbf{\emph{i}}',\textbf{\emph{j}}') \Longleftrightarrow   wt(\textbf{\emph{i}},\textbf{\emph{j}})=(\textbf{\emph{i}}',\textbf{\emph{j}}'),$$
\end{remark}
and the map $wt$ is defined in subsection \ref{comb}.
\begin{remark}\label{remark}
The elements of $\Lambda(n,n;r)$ can be realized as
$n\times n$ matrices of non-negative integers $(\omega_{st})_{s,t}$
with $1\leq s,t\leq n$ such that $\sum_{s,t=1}^n\omega_{st}=r$.

Denote by $\Lambda^s(n,n;r)$ the subset of $\Lambda(n,n;r)$ as
\begin{eqnarray*}
\Lambda^s(n,n;r)=\{\omega\in\Lambda(n,n;r)|\omega_{ij}=0\   \forall\   i>j,  \sum\limits_{1\leq k\leq l\leq n}(l-k)\omega_{kl}\geq s\}.
\end{eqnarray*}
Under the identification in \ref{combiso}, we simply find that
\begin{enumerate}
\item $(\textbf{\emph{i}},\textbf{\emph{j}})$ satisfies that $\textbf{\emph{i}}\geq\textbf{\emph{j}}$
   $\Longleftrightarrow$ Put $(\omega_{ij})= wt(\textbf{\emph{i}},\textbf{\emph{j}})$, which means $\omega_{ij}=0\   \forall\   i>j$.

\item Put $(\omega_{ij})= wt(\textbf{\emph{i}}_\lambda d,\textbf{\emph{j}}_\mu)$. There exists a pair $i,j$ such that $\omega_{ij}\neq 0$ $\Longrightarrow$ $\lambda \neq \mu$.
\end{enumerate}
Thus, we can set two definition of subset as following
\begin{eqnarray*}
 \Omega^{\succeq0}:=\Xi^{-1}(\Lambda^0(n,n;r))=\bigsqcup\limits_{\lambda,\mu\in\Lambda(n,r)}\{d\in \mathscr{D}_{\lambda\mu}|  \textbf{\emph{i}}_\lambda d\geq\textbf{\emph{i}}_\mu\}
 :=\bigsqcup\limits_{\lambda,\mu\in\Lambda(n,r)}\Omega_{\lambda\mu}^{\succeq0}.
\end{eqnarray*}
\begin{eqnarray*}
\Omega^{\succeq 1}:=\Xi^{-1}(\Lambda^1(n,n;r))=\{d\in \Omega^{\succeq0}|\lambda,\mu\in \Lambda(n,r),\  \lambda\neq\mu\  and \
d\neq 1\}.
\end{eqnarray*}
Similarly, we can defined a sequence of sets as
$\Omega^\succeq\supseteq\Omega^{\succeq1}\supseteq\Omega^{\succeq2}\supseteq\cdots\Omega^{\succeq
n}\supseteq\cdots$, and defined subspaces of
$S_{R}(n,r)$:
\begin{eqnarray*}
J_n:=\bigoplus\limits_{d\in\Omega^{\succeq
n}}R\psi_{\lambda\mu}^d\subseteq S_{R}^+(n,r) \qquad with\
 d\in\mathscr{D}_{\lambda\mu}.
\end{eqnarray*}
We set $S^+_{R}(n,r):=J_0$, it is {\em{Borel subalgebra}} of
$q$-Schur algebra. This definition is the same as \cite{1}.
$J_1=\bigoplus\limits_{
\psi_{\lambda\mu}^d\neq\psi_{\lambda\lambda}^1}R\psi_{\lambda\mu}^d$
and $J_0\supseteq J_1\supseteq J_2\supseteq\cdots$
\end{remark}
Moreover, if we define subsets of $\Omega^{\succeq n}$ as $\Omega^{\succeq n}_{\lambda\mu}:= \Omega^{\succeq n}\cap \mathscr{D}_{\lambda\mu}$, then, we find trivially that
$\bigsqcup\limits_{\lambda,\mu\in\Lambda(n;r)}\Omega_{\lambda\mu}^{\succeq n}=\Omega^{\succeq n}$. Using this notations, we can state the following lemmas.

\begin{lemma}\label{lemma}
If $d_1\in \Omega^{\succeq n}_{\lambda\mu}$ and $d_2\in
\Omega^{\succeq m}_{\mu\nu}$, then
$\psi_{\lambda\mu}^{d_1}\psi_{\mu\nu}^{d_2}=\sum\limits_{d\in\Omega_{\lambda\nu}^{\succeq
m+n}}a_d\psi_{\lambda\nu}^d$ for some $a_d\in R$.

\end{lemma}
\begin{proof}
Suppose $\lambda,\mu\in\Lambda(n,r),d\in\Omega^{\succeq
n}_{\lambda\mu}\subseteq\mathscr{D}_{\lambda\mu}$. We claim that
$\psi_{\lambda\mu}^d(X_{\textbf{i}\textbf{j}})\neq 0$ implies
$(\textbf{i},\textbf{j})\in I(n,r)\times I(n,r)$ satisfies
$\sum\limits_{1\leq l\leq k\leq
n}(l-k)\omega_{lk}=\sum\limits_{k=1}\limits^r(i_k-j_k)\geq n$.

Indeed, by the definition above, we have hypothesis that
$\textbf{j}=\textbf{i}_\mu w$ for some $w\in\mathfrak{S}_r$. If
$\ell(w)=0$, i.e., $w=1$, then $\textbf{j}=\textbf{i}_\mu$ and
$\textbf{i}=\textbf{i}_\lambda d$, which trivially shows
$(\textbf{i},\textbf{j})\in \Omega^{\succeq n}_{\lambda\mu}$.

Assume now $\ell(w)>0$. Write $w=w't$ with $t=(a,a+1)$ and
$\ell(w)=\ell(w')+1$. Then by definition of $\textbf{i}_\mu$, we
have $j_a>j_{a+1}$. If $i_a\leq i_{a+1}$, then by \ref{sdual},
$X_{\textbf{i\textbf{j}}}=qX\textbf{i}t,\textbf{j}t=qX\textbf{i}t,\textbf{i}_{\mu
w'}$. By induction on the length of elements in $\mathfrak{S}_r$,
$\psi_{\lambda\mu}^d(X_{\textbf{i\textbf{j}}})=q\psi_{\lambda\mu}^d(X\textbf{i}t,\textbf{i}_{\mu
w'})\neq 0$ shows $(\textbf{i}t,\textbf{i}_\mu w')$ satisfy above
claim, which trivially implies $(\textbf{i},\textbf{j})$ satisfy
this claim too.

If $i_a>i_{a+1}$, then also by \ref{sdual} relations.
\begin{eqnarray*}
X_{i_aj_a}X_{i_{a+1}j_{a+1}}&=&X_{i_{a+1}j_{a+1}}X_{i_aj_a}-(q^{-1}-1)X_{i_{a+1}j_a}X_{i_aj_{a+1}}\\
&=&X_{i_{a+1}j_{a+1}}X_{i_aj_a}-(1-q)X_{i_aj_{a+1}}X_{i_{a+1}j_a}
\end{eqnarray*}
Thus
$\psi_{\lambda\mu}^d(X_{\textbf{i}\textbf{j}})=\psi_{\lambda\mu}^d(X_{\textbf{i}t,\textbf{j}t})-(1-q)\psi_{\lambda\mu}^d(X_{\textbf{i},\textbf{j}t})\neq
0$. We have either
$\psi_{\lambda\mu}^d(X_{\textbf{i}t,\textbf{j}t})\neq0$ or
$\psi_{\lambda\mu}^d(X_{\textbf{i},\textbf{j}t})\neq 0$. By
induction, either $(\textbf{i}t,\textbf{j}t)$ or
$(\textbf{i},\textbf{j}t)$ satisfy above claim, which trivially
shows that $(\textbf{i},\textbf{j})$ satisfies this claim too.

Using the multiplication rules in $A_q(n,r)^*$, we have following:
\begin{eqnarray*}
 \psi_{\lambda\mu}^{d_1}\cdot\psi_{\mu\nu}^{d_2}(X_{\textbf{i}\textbf{j}})
&=&<\psi_{\lambda\mu}^{d_1}\otimes\psi_{\mu\nu}^{d_2},\triangle(X_{\textbf{i}\textbf{j}})>\\
&=&\sum\limits_{\textbf{k}\in
I(n,r)}\psi_{\lambda\mu}^{d_1}(X_{\textbf{i}\textbf{k}})\psi_{\mu\nu}^{d_2}(X_{\textbf{k}\textbf{j}})\neq0.
\end{eqnarray*}
So there is $\textbf{k}\in I(n,r)$ such that
$\psi_{\lambda\mu}^{d_1}(X_{\textbf{i}\textbf{k}})\neq0$ and
$\psi_{\mu\nu}^{d_2}(X_{\textbf{k}\textbf{j}})\neq0$. Thus,
$\sum\limits_{h=1}\limits^r(i_h-k_h)\geq n$ and
$\sum\limits_{l=1}\limits^r(k_l-j_l)\geq m$, and hence
$\sum\limits_{s=1}^r(i_s-j_s)=\sum(i_h-k_h)-\sum(k_l-j_l)\geq n+m$.

And we only need to show that $a_d\neq 0$ only if
$\textbf{i}_\lambda\succeq\textbf{i}_\nu$, which has already been
done in \cite{1} by Du and Rui.

\end{proof}

\begin{proposition}\label{radical}
The subspaces $\{J_m\}$ with $m\in\mathbb{N}_0$ are ideals of
$S_{R}^+(n,r)$. Moreover, nilpotent ideal $J_1$ is actually the
radical of $S_{R}^+(n,r)$ when $R$ is a field, which is spanned by
$\{\psi_{\lambda\mu}^d|\lambda\neq\mu,\ d\neq1\}$.

\end{proposition}

\begin{proof}
\begin{eqnarray*}\label{multip}
 \psi_{\lambda\mu}^{d_1}\cdot\psi_{\omega\nu}^{d_2}(X_{\textbf{i}\textbf{j}})
&=&<\psi_{\lambda\mu}^{d_1}\otimes\psi_{\omega\nu}^{d_2},\triangle(X_{\textbf{i}\textbf{j}})>\\
&=&\sum\limits_{k\in
I(n,r)}\psi_{\lambda\mu}^{d_1}(X_{\textbf{i}\textbf{k}})\psi_{\omega\nu}^{d_2}(X_{\textbf{k}\textbf{j}}).
\end{eqnarray*}
Using the formula above,
$\psi^d_{\lambda\mu}(X_{\rho\nu}^{d_1})=\delta_{\lambda\rho}\delta_{\mu\nu}\delta_{d,d_1}$,
we can claim that: if that $\mu\neq\omega$, then
$\psi_{\lambda\mu}^{d_1}\cdot\psi_{\omega\nu}^{d_2}(X_{\textbf{i}\textbf{j}})=
0$. Then by the consequence of Lemma \ref{lemma}, we find subspace
$J_n:=\bigoplus\limits_{d\in\Omega^{\succeq
n}}R\psi_{\lambda\mu}^d$ is ideal of $S_{R}^+(n,r)$, since
$J_0J_n\subseteq J_n$ with $J_0=S_{R}^+(n,r)$.

Let
$L_{n,r}:=\bigoplus_{\lambda\in\Lambda(n,r)}R\psi_{\lambda\lambda}^1$,
then $L_{n,r}$ is a commutative $R$-subalgebra of $S_{R}^+(n,r)$,
and $S_{R}^+(n,r)=L_{n,r}\oplus J_1$. For every $\lambda\in
\Lambda(n,r)$ we have a $R$-free module
$R_\lambda:=R\psi_{\lambda\lambda}^1$ of rank one. Note that
$\psi_{\lambda\lambda}^1$ acts on $R_\lambda$ as identity, and
$\psi_{\mu\mu}^1,\  \mu\neq\lambda$, acts as zero. We will denote in
the same way the module over $S_{R}^+(n,r)$ obtain from $R_\lambda$
by the natural projection of $S_{R}^+(n,r)$ on $L_{n,r}$.

Note that if $R$ is a field the algebra $L_{n,r}$ is semi-simple,
and so $J_1$ is the radical of $S_{R}^+(n,r)$. In this case,
$\{R_\lambda|\lambda\in\Lambda(n,r)\}$ is a complete set of pairwise
non-isomorphic simple modules over $S_{R}^+(n,r)$.
\end{proof}

\begin{remark}\label{triple}
By the result of \ref{radical}, in the category of rings, we can
construct a splitting map $p$: $S^+_R(n,r) \rightarrow$
$\bigoplus\limits_{\lambda\in\Lambda(n,r)}R\psi_{\lambda\lambda}^1$
of the including map $i$:
$\bigoplus\limits_{\lambda\in\Lambda(n,r)}R\psi_{\lambda\lambda}^1$
$\hookrightarrow$ $S^+_R(n,r)$.\\
From now on, denote $A:= S^+_R(n,r)$,
$J:= rad(S^+_R(n,r))$ and
$S:=\bigoplus\limits_{\lambda\in\Lambda(n,r)}R\psi_{\lambda\lambda}^1$.
Then, we define a homomorphism of $S$-bimodules as $\widetilde{p}$:
$A\rightarrow J$ with $a\mapsto a-p(a)$. Obviously $S$ is a
commutative ring and the restriction of $\widetilde{p}$ to $J$ is
the identity map.

\end{remark}

\begin{definition}\label{bar}
For every left A-module M, define a complex $\mathbb{C}_*(A,S,M)$ whose factors denoted by $C_k(A,S,M)$, where integer
$k\geq-1$, following the notation of in Remark \ref{triple}. We set the several
notations:
\begin{equation}
\left\{
\begin{array}{ccc}
C_{-1}(A,S,M)=M&&k=-1\\
C_0(A,S,M)=A\otimes M&&k=0\\
C_{k}(A,S,M)=A\otimes J^{\otimes k}\otimes M&&k>1\\
\end{array}
\right.
\end{equation}
where all the tensor products are taken over $S$.\\
Next we define $A$-module homomorphism
$d_{k,j}:C_k(A,S,M)\rightarrow C_{k-1}(A,S,M)$, $0\leq j\leq k$, and
$S$-module homomorphisms $s_k$:$C_k(A,S,M)\rightarrow
C_{k+1}(A,S,M)$ by:
\begin{eqnarray*}
d_{0,0}(m)&:=&am,\\
d_{k,0}(a\otimes a_1\otimes\cdots\otimes a_k\otimes
m)&:=&aa_1\otimes a_2\otimes\cdots\otimes a_k\otimes m,\\
d_{k,j}(a\otimes a_1\otimes\cdots\otimes a_k\otimes m)&:=&a\otimes
a_1\otimes\cdots\otimes a_ja_{j+1}\otimes\cdots\otimes a_k\otimes m,
\  1\leq j\leq k-1,\\
d_{k,k}(a\otimes a_1\otimes\cdots\otimes a_k\otimes
m)&:=&a\otimes a_1\otimes\cdots\otimes a_{k-1}\otimes a_km,\\
s_{-1}(m)&:=&e\otimes m,\\
s_k(a\otimes a_1\otimes\cdots\otimes a_k\otimes m)&:=&e\otimes
\widetilde{p}(a)\otimes a_1\otimes\cdots\otimes a_k\otimes m, 0\leq
k.
\end{eqnarray*}
Define an A-module homomorphism $d_k$: $C_k(A,S,M)\rightarrow C_{k-1}(A,S,M)$ by:
$$d_k:=\sum\limits^{k}\limits_{t=0}(-1)^td_{k,t}$$.
\end{definition}

\begin{proposition}\cite{2}
The sequence $(C_k(A,S,M),d_k)_{k\geq -1}$ is a complex of left
$A$-module. Moreover, we have the relations:
\begin{eqnarray*}
d_0s_{-1}&=&id_{C_{-1}(A,S,M)},\\
d_{k+1}s_k+s_{k-1}d_k&=&id_{C_k(A,S,M)}, \  0\leq k.
\end{eqnarray*}
Thus, these maps $\{s_k\}_{k\geq1}$ give a splitting of $\mathbb{C}_*(A,S,M)$
in the category of $S$-modules, i.e,. Prove that $id_{\mathbb{C}_*}$ a zero holomorphic chain map. In particular,
$(C_k(A,S,M),d_k)_{k\geq -1}$ is exact.
\end{proposition}

\begin{definition}\label{involution}
Let $\xi$ be the anti-automorphism of the $q$-Schur algebra
$S_R(n,r)$, which is defined as
\begin{eqnarray*}
S_R(n,r)\rightarrow S_R(n,r) \qquad
 \psi_{\lambda\mu}^d\mapsto\psi_{\mu\lambda}^{d^{-1}}
\end{eqnarray*}
It is clear that with definition of $S_R(n,r)$, $S^+_R(n,r)$ and
$S^-_R(n,r):=\xi(S^+_R(n,r))$. This anti-automorphism $\xi$ can be
restricted as a ring homomorphism $S^+_R(n,r)\rightarrow
S^-_R(n,r)$. Following the definition in Appendix 7 of \cite{3}, we
define a functor $\mathcal {J}$ which is called a \emph{contravariant} \emph{dual functor}
 of the left $S_R(n,r)$-module category $_{S_R(n,r)}{\bf mod}$:
\begin{eqnarray*}
_{S_R(n,r)}{\bf mod}&\rightarrow& _{S_R(n,r)}{\bf mod}\\
V&\mapsto&V^\varoast
\end{eqnarray*}
where the left module structure is defined to satisfy that
$(s\theta)(v)=\theta(\xi(s)v)$, for $\theta\in Hom_R(V,R)$,
$s\in S_R(n,r)$, $v\in V$.
\end{definition}

Then we may consider the right exact functor
\begin{eqnarray*}
F=S_R(n,r)\otimes_{S^-_R(n,r)}-: \quad
{}_{S^-_R(n,r)}{\bf mod}\rightarrow {}_{S_R(n,r)}{\bf mod},
\end{eqnarray*}
and the left exact functor
\begin{eqnarray*}
G=Hom_{S^+_R(n,r)}(S_R(n,r),-): \quad
{}_{S^+_R(n,r)}{\bf mod}\rightarrow {}_{S_R(n,r)}{\bf mod}.
\end{eqnarray*}
\begin{lemma}\label{acyclic}
With the notation above, there is a $S_R(n,r)$-isomorphism
$$F(V^\varoast )\cong(G(V))^\varoast $$
naturally in $V\in S^+_R(n,r)\text{-mod}$
\end{lemma}
\begin{proof}
In \cite{3}(section 7), the author proved a more general result for
these algebras.
\end{proof}

Let $R_\Lambda:=\bigoplus\limits_{\lambda\in\Lambda}R\psi_\lambda$, with $\psi_\lambda:=\psi_{\lambda\lambda}^1$. Then $R_\Lambda$ is a commutative $R$-subalgebra of $S^+_R(n,r)$, and
$S^+_R(n,r)=R_\Lambda\oplus J_1$. For every $\lambda\in\Lambda(n,r)$ we have a free $R$-module module $L_\lambda:=R\psi_\lambda$ of rank one over $R_\Lambda$. Note that $\psi_\lambda$ acts on $L_\lambda$
by identity, and $\psi_\mu$, $\mu\neq\lambda$, acts by zero. We will denote in the same way the module over $S^+_R(n,r)$ obtained from $L_\lambda$ by inflating along the natural projection
of $S^+_R(n,r)$ on $R_\Lambda$, i.e., $J_1$ and $\psi_\mu$($\mu\neq\lambda$) act on $L_\lambda$ as zero, and $\psi_\lambda$ as identity.

Note that if $R$ is a field then algebra $R_\Lambda$ is semi-simple, and so $J_1$ is the radical of $S^+_R(n,r)$. In this case $\{L_\lambda|\lambda\in\Lambda(n,r)\}$ is a complete set of
pairwise non-isomorphic simple modules over $S^+_R(n,r)$.

For $\lambda\in\Lambda(n,r)$ we denote the resolution $C_*(S^+_R(n,r),R_\Lambda,L_\lambda)$ defined in \ref{bar} by $\mathbb{C}_*^+(L_\lambda)$, and call it \emph{bar resolution}.
The factor $\mathbb{C}_k^+(L_\lambda)$ in resolution $C_*(S^+_R(n,r),R_\Lambda,L_\lambda)$ has following form:
\begin{eqnarray}
S^+_R(n,r)\otimes J_1\otimes\cdots\otimes J_1\otimes L_\lambda,
\end{eqnarray}
where all tensor products are over commutative ring $L_\lambda$ and there are $k$ factors $J_1$.

\begin{proposition}\label{iso}
Let $\nu,\mu\in\Lambda$ and $n\geq0$. Then $\psi_\nu J_n\psi_\mu=0$ unless $\nu\triangleright\mu$ (which means $\nu\unrhd\mu$ but $\nu\neq\mu$).
If $\nu\triangleright\mu$, then
\begin{eqnarray*}
\{\psi^d_{\nu\mu}|d\in \Omega^{\succeq n}_{\nu\mu}\}
\end{eqnarray*}
is an $R$-basis of the free $R$-module $\psi_\nu J_n\psi_\mu$.
\end{proposition}

\begin{proof}
From Remark \ref{remark} it follow that the set $J_n=\bigoplus\limits_{\substack{d\in\Omega^{\succeq n}_{\theta\eta}\\
\theta,\eta\in\Lambda(n,r)}}R\psi^d_{\theta\eta}$. Then, we have $\psi_\nu J_n\psi_\mu=\psi_{\nu\nu}^1\cdot(\bigoplus\limits_{\substack{d\in\Omega^{\succeq n}_{\theta\eta}\\
\theta,\eta\in\Lambda(n,r)}}R\psi^d_{\theta\eta})\cdot\psi_{\mu\mu}^1=\bigoplus\limits_{d\in \Omega_{\nu\mu}^{\succeq n}}R\psi_{\nu\mu}^d$. The first statement follows from
the fact $\Omega_{\nu\mu}^{\succeq n}\subseteq \Omega_{\nu\mu}^{\succeq} $ and $d\in\Omega_{\nu\mu}^{\succeq}$ shows that $\nu\rhd\mu$.

\end{proof}

\begin{proposition}\label{simple}
For all $\lambda\in \Lambda(n,r)$, We have $\mathbb{C}_0^+(L_\lambda)\cong
S^+_R(n,r)\psi_\lambda$, and when any $k\geq 1$, factor of $\mathbb{C}_k^+(L_\lambda)$ has the form
\begin{eqnarray*}
\bigoplus\limits_{\mu^{(1)}\rhd\cdots\rhd\mu^{(k)}\rhd\lambda}S^+_R(n,r)\psi_{\mu^{(1)}}\otimes_R\psi_{\mu^{(1)}}J_1\psi_{\mu^{(2)}}\otimes_R\cdots\otimes_R\psi_{\mu^{(k)}}J_1\psi_\lambda,
\end{eqnarray*}
\end{proposition}
\begin{proof}
First of all, let $M$ be a right $R_\Lambda$-module and $N$ a left $R_\Lambda$-module. It follows from Corollary 9.3 in \cite{12} that $M\otimes_{R_\Lambda}N\cong\bigoplus_{\lambda\in\Lambda(n,r)}M\psi_\lambda\otimes_{R}\psi_\lambda N$.

 Then from the above statement, we can tell that $\mathbb{C}_k^+(L_\lambda)$ is the direct sum of $S^+_R(n,r)$-modules like:
\begin{eqnarray*}
S^+_R(n,r)\psi_{\mu^{(1)}}\otimes \psi_{\mu^{(1)}}J_1\psi_{\mu^{(2)}}\otimes\cdots\otimes\psi_{\mu^{(k)}}J_1\psi_{\mu^{(k+1)}}\otimes\psi_{\mu^{(k+1)}}L_\lambda.
\end{eqnarray*}
where all tensor products are over $R$ and the sum is taken on any sequence $(\mu^{(1)},\ldots, \mu^{(k+1)})\in\big(\Lambda(n,r)\big)^{k+1}$ .
With the property $\psi_{\mu^{(k+1)}}L_\lambda=0$ unless $\mu^{(k+1)}=\lambda$ and the consequence of Proposition \ref{simple}, the summation is in fact over the sequences $\mu^{(1)}\triangleright\ldots\triangleright\mu^{(k)}\triangleright\lambda$.
\end{proof}

\begin{proposition}\label{finitel}
Let $\lambda\in\Lambda(n,r)$. Then $\mathbb{C}_*^+(L_\lambda)$ is a projective resolution of the module $L_\lambda$ over $S^+_R(n,r)$ with finite length .

\end{proposition}

\begin{proof}
Let $N$ be the length of the maximal strictly decreasing sequence in $(\Lambda(n,r),\rhd)$. Then $\mathbb{C}_k^+(L_\lambda)=0$ for $k>N$ by Proposition \ref{iso} and \ref{simple}.
Therefore, we can tell that $\mathbb{C}_*^+(L_\lambda)$ is complex with finite length.

The summand $S^+_R(n,r)\psi_{\mu^{(1)}}\otimes_R\psi_{\mu^{(1)}}J_1\psi_{\mu^{(2)}}\otimes_R\cdots\otimes_R\psi_{\mu^{(k)}}J_1\psi_\lambda$ of $\mathbb{C}_*^+(L_\lambda)$
is a projective $S^+_R(n,r)$-module, since $S^+_R(n,r)\psi_{\mu^{(1)}}$ is projective and every $\psi_{\mu^{(i)}}J_1\psi_{\mu^{(i+1)}}$ is isomorphic to a free $R$-module with finite rank. Then, we can say that $\mathbb{C}_k^+(L_\lambda)$ is a projective $S^+_R(n,r)$-module and $\mathbb{C}_*^+(L_\lambda)$ is a projective resolution of the module $L_\lambda$ over $S^+_R(n,r)$.

\end{proof}

\section{\bf Woodcock's condition and Kempf's Vanishing theorem}

In this section we explain a condition which appears in Kashiwara's work \cite{13}
and Woodcock's work \cite{13}, and explain how the result
of Kempf's Vanishing theorem can be applied to prove that
$S^+_R(n,r)$-module $L_\lambda$ is
acyclic for the induction functor $S_R(n,r)\otimes_{S_R^+(n,r)}-$.

Let $P$ be a set and $\lambda\in P$, and let $R(\lambda)$ be a copy of the trivial $R$-coalgebra. Let $1_\lambda$ be the element of $\coprod_{\lambda\in P}R(\lambda)$ whose $\mu$-component is $\delta_{\lambda,\mu}$.
Then, for any left (resp. right) comodule $(V,\rho)$ of $\coprod_{\lambda\in P} R(\lambda)$,
 we have $V=\otimes_{\lambda\in P} {}^\lambda V$ (resp. $\otimes_{\lambda\in P}V^\lambda$),
 where $^\lambda V:=\{v\in V|\rho(v)=1_\lambda\otimes v\}$ (resp. $V^\lambda := \{v\in V|\rho(v)=v\otimes 1_\lambda\}$)
 is called left (resp. right) \emph{weight space} for weight $\lambda$.

Suppose we are given a partical order $\leq$ on a poset $P$ and a
subset $P^+$ of $P$ which is a \emph{locally finite} poset, i.e.,
 for each $\lambda\in P$ there are only finitely many $\mu\in P^+$ with $\mu\leq\lambda$.

 Let $(A(\Lambda),\mu^\Lambda_\Gamma)$ be a filtered system of $R$-coalgebras indexed by the
 finite ideals in $P^+$. Assume that each map $\mu^\Lambda_\Gamma:A(\Gamma)\rightarrow A(\Lambda)$ is injective,
and maps $A(\Lambda)\rightarrow \coprod_{\lambda\in P} R(\lambda)$ compatible
 with $\mu^\Gamma_\Lambda$.
 Put $C=\varinjlim\limits_{\Lambda}A(\Lambda)$.
\begin{definition}
 we say that $(A(\Lambda),P^+,C)$ satisfy a \emph{Woodcock condition}:
  If we have an isomorphism of bicomodules
 \begin{eqnarray*}
 A(\Lambda)/A(\Lambda\setminus\{\lambda\})\cong\triangledown(\lambda)\otimes\triangledown'(\lambda).
 \end{eqnarray*}
where $\triangledown(\lambda)$ and $\triangledown'(\lambda)$ are, respectively, left and right $A(\lambda)$-comodules satisfying:
\begin{enumerate}
\item  $^\lambda\triangledown(\lambda)\cong\triangledown'(\lambda)^\lambda\cong R(\lambda)$.

\item  For all $\lambda\in P$, $^\mu\triangledown(\lambda)\neq0$ or $\triangledown'(\lambda)^\mu\neq0$ implies $\mu\leq\lambda$.

\item $\triangledown(\lambda)$ and $\triangledown'(\lambda)$ are finitely generated and projective over $R$.
\end{enumerate}
\end{definition}

The most valuable example of Woodcock condition appears in quantized enveloping algebra
and the \emph{coordinate algebra} corresponding to it.

\begin{example}\cite{14}\label{kashi}
Let $U$ be a quantized enveloping algebra.

$A:=\{a\in U^* | Ua,  aU \text{are both integrable}\}$
has a
$\mathbb{Q}(q)$-coalgebra structure where the comultiplication is
defined by $\triangle(c)(u\otimes v)=c(vu)$ and the counit by
$\varepsilon(c)=c(1)$, for any $c\in A$ and $ u,v\in U$.

In \cite{14}, Woodcock showed that,
 if $\Lambda$ is a finite ideal in $P^+$ and $V\in U$-int, let $O_\Lambda V$ be the largest submodule $V'$ of $V$ such that $^\lambda V'\neq0$, $\lambda\in P^+$ implies $\lambda\in \Lambda$.
Put $A(\Lambda):=O_\Lambda A$, we know that $(A(\Lambda),P^+,A)$ is a triple which satisfies Woodcock's condition.
\end{example}

\begin{remark}\cite{14}
Let $\Lambda$ be a finite ideal of $P^+$. Denote $A_R(\Lambda)=R\otimes_\mathscr{A} A(\Lambda)$ and $A_R=R\otimes_\mathscr{A} A$, where
$A(\Lambda)$ is  defined as above. Moreover, let $\lambda$ be the maximal one in
$\Lambda$ and $\Gamma=\Lambda\setminus\{\lambda\}$, then we have a short exact sequence in $U$-bimodule
\begin{eqnarray}\label{exact1}
0\rightarrow A_R(\Gamma)\rightarrow A_R(\Lambda)\rightarrow \triangledown_R(\lambda)\otimes \triangledown_R(\lambda)'\rightarrow 0,
\end{eqnarray}
which maps those global basis elements in $A_R(\Lambda)\setminus A_R(\Gamma)$ bijectively onto the standard global basis of $\triangledown(\lambda)\otimes\triangledown(\lambda)'$.

Therefore, the triple $(A_R(\lambda),P^+,A_R)$ satisfies the Woodcock condition defined above, where the ground ring $R$ has a ring morphism $\mathbb{Z}[q,q^{-1}]\rightarrow R$.
\end{remark}

\begin{definition}\label{dmodule}\cite{13}
For $\mu\in W\cdot\lambda$ with $\lambda\in P^+$, $W$ is the weyl group associated with root datum of $U$.
Let $\nu_\mu$ be the element of global basis of $\triangle(\lambda)$ with weight $\mu$. If $w\in W$ with $m=<\alpha^\vee,w\lambda>\geq0$, we have
\begin{eqnarray*}
\nu_{s_iw\lambda}&=&f^{(m)}_i\nu_{w\lambda},
\qquad \qquad\nu_{w\lambda}=e_i^{(m)}\nu_{s_iw\lambda}.
\end{eqnarray*}
Put $\triangle^\#(\mu):=U^\#\cdot\nu_\mu$. They are the \emph{Demazure modules} associated to $\mu$.
\end{definition}

\begin{example}\cite{14}
Dualizing short exact sequence (\ref{exact1}) and induction on $|\Lambda|$, we give an algebra epimorphism $U_R\rightarrow S_R$, $u\mapsto u\cdot 1$. Write $S^\mathfrak{b}_R$ for the image of $U^\mathfrak{b}_R$ under this map.

The Borel Schur algebras still fit the framework of Woodcock's condition:

 Let $\Xi\leq P$ be  a finite ideal for the \emph{antipodal excellent order}, then $W\Xi\cap P^+$ is a finite ideal in $P^+$. Thus we may choose $W\Xi\cap P^+$ to be the ideal $\Lambda$.
Let $F_\Xi: {}_{S^b_R}{\bf mod}\rightarrow {}_{S^b_R}{\bf mod}$ take $V$ to its largest quotient with weight in $\Xi$, a right exact functor.

Put $S^\mathfrak{b}_R(\Xi)=F_\Xi S^\mathfrak{b}_R$. By the
alternative version in \cite{18},
 if $\mu$ is maximal in $\Xi$ for the antipodal excellent order, then there is a short exact sequence of bimodules
 \begin{eqnarray}
 0\rightarrow\triangle^\mathfrak{b}_R(\mu)\otimes\bar{\triangle}^{\#}_R(\mu)^\tau\rightarrow S^\mathfrak{b}_R(\Xi)\rightarrow S^\mathfrak{b}_R(\Xi\setminus\{\mu\})\rightarrow0.
 \end{eqnarray}

By induction on $|\Xi|$ it now can be shown that $S^\mathfrak{b}_R$ is finitely generated free over $R$. Therefore, the $A^\mathfrak{b}_R(\Xi):=(S^\mathfrak{b}_R(\Xi))^*$ forms a filtered system of $R$-coalgebras.
Write $C^\mathfrak{b}_R$ as the colimit of this system, which is a free $R$-coalgebra.

 Suppose one takes $P$ for both $P$ and $P^+$, $\preceq$ for order $\leq$, and defines the maps $A^\mathfrak{b}_R\rightarrow \coprod_{\lambda\in P}R(\lambda)$ by using the idempotent $\varepsilon_\mu$. Then the triple $(A^\mathfrak{b}_R, P, C^\mathfrak{b}_R)$ satisfies the Woodcock's condition. It implies that the roles of the $\triangle(\mu)$ and $\triangle'(\mu)$
are respectively played by $\triangle^\mathfrak{b}_R(\mu)$ and $\bar{\triangle}^{\#}_R(\mu)'$.
\end{example}

The triple $(A(\Lambda),P^+,C)$ with Woodcock condition have some
very interesting homological consequences \cite{14}. One of them is
called \emph{Ext-reciprocity}. For $V\in S(\Lambda)\text{-Mod}$,
$X\in R\text{-Mod}$, and $\mu\in P$, we have
\begin{eqnarray}\label{ext}
Ext^i_{S(\Lambda)}(S(\Lambda)^\mu\otimes X, V)\cong Ext^i_R(X, ^\mu V)\qquad  \forall i \geq 0,\\
Ext^i_{S(\Lambda)}(V,A(\Lambda)^\mu\otimes X)\cong Ext^i_R(^\mu V,X)\qquad  \forall i \geq 0.
\end{eqnarray}

A $q$-analogue Kempf's vanishing theorem was established by some
properties of the crystal basis proved by Kashiwara in order to
obtain the refined Demazure character formula in \cite{13}\cite{22}.

Woodcock used that the ideal of using the properties of
the cystal basis to obtain a quantized Kempf's vanishing theorem
also works for Schur algebra's version.
Let us recall the {\em{Kempf's vanishing theorem}} in Woodcock \cite{14}:

For any $R$-module $X$ and $\mu\in P$, note that $\mu^+$ as the unique dominant
weight in orbit $W\mu$. We have
\begin{eqnarray*}\label{vanishing}
Ext^i_{S^b_R}(S_R,\triangledown^\mathfrak{b}_R(\mu)\otimes X)&\cong& Ext^i_R(^\mu S_R, X)\cong Ext^i_R(\triangle_R(\mu^+)^\tau, X)\\
&\cong&\left\{
\begin{array}{c}
\triangle_R(\mu^+)\otimes X \qquad if\  i=0\\
0  \qquad \qquad \qquad \quad if\   i>0
\end{array}
\right.
\end{eqnarray*}

\begin{remark}
The case $\mu=\mu^+$ of \ref{vanishing} is Kempf's vanishing theorem in \cite{22},
 since $H^i(U_R/U^\mathfrak{b}_R)\cong Ext^i_{S^{b}_R(W\Lambda)}(S_R(W\Lambda),\triangle(\lambda))$ for all $i\geq0$.
If $\Lambda$ is a finite ideal in $P^+$, and $V\in {}_{S^b_R(W\Lambda)}\bf{mod}$.

Moreover, in this situation, we have $\triangledown^\mathfrak{b}_R(\mu)= L_\mu$, and $S^{\mathfrak{b}}_R=S^-_R$. Thanks to the definition of Demazure modules in \ref{dmodule},
which means the module
$L_\lambda$ is $Hom_{S^-_R}(S_R,-)$-acyclic.
\end{remark}

\begin{theorem}\label{main1}
 For $\lambda\in \Lambda^+(n,r)$, the complex $ S_R(n,r)\otimes_{S^+_R(n,r)} \mathbb{C}_*^+(L_\lambda)$ is a projective resolution of $W^L_\lambda:=S_R(n,r)\otimes_{S^+_R(n,r)}L_\lambda$ over $S_R(n,r)$.
\end{theorem}

\begin{proof}
 Given $\lambda\in \Lambda^+(n,r)$,  denote the complex $\mathbb{B}(R):=S_R\otimes_{S^+_R} \mathbb{C}_*^+(L_\lambda)$.
Since  the $S_R$-module isomorphism holds:
\begin{eqnarray*}
\mathbb{B}(R) \cong\bigoplus\limits_{\mu^{(1)}\rhd\cdots\rhd\mu^{(k)}\rhd\lambda}S_R\psi_{\mu^{(1)}}\otimes_R\psi_{\mu^{(1)}}J_1\psi_{\mu^{(2)}}\otimes_R\cdots\otimes_R\psi_{\mu^{(k)}}J_1\psi_\lambda,
\end{eqnarray*}
and $S_R\psi_{\mu^{(1)}}$, $\psi_{\mu^{(i)}}J_1\psi_{\mu^{(i+1)}}$ with $1\leq i\leq k$ and $\psi_{\mu^{(k)}}J_1\psi_\lambda$ are all free $R$-module by Proposition \ref{iso},   it follows that all factors in $\mathbb{B}(R)$ are free $R$-modules.

The short exact sequence
$$0\rightarrow H_i(\mathbb{B}(\mathbb{Z}))\otimes R\rightarrow H_i(\mathbb{B}(R))\rightarrow Tor^\mathbb{Z}_1(H_{i-1}(\mathbb{B}(\mathbb{Z})), R)\rightarrow 0$$
follows from  the Universal Coefficient Theorem over the complex
$\mathbb{B}(\mathbb{Z})\otimes_\mathbb{Z}R\cong \mathbb{B}(R)$.

Hence, in order to show that the complex $\mathbb{B}(R)$ are acyclic, it is enough to prove the claim that
the complex $\mathbb{B}(\mathbb{Z})$ is acyclic.

We already know that $H_i(\mathbb{B}(\mathbb{Z}))$ is a finitely generated abelian group. Therefore
we can write $H_i(\mathbb{B}(\mathbb{Z}))$ as $\mathbb{Z}^\alpha\bigoplus\limits_{p\text{ is a }prime}\bigoplus\limits_{s\geq1}(\mathbb{Z}/p^s\mathbb{Z})^{\alpha_{ps}}$,
where only finitely many of the integer $\alpha, \alpha_{ps}$ are different from zero. For every prime $p$ denote by $\bar{\mathbb{F}}_p$ the algebraic closure $\mathbb{F}_p$.
We get
$
H_i(\mathbb{B}(\mathbb{Z}))\otimes_\mathbb{Z}\bar{\mathbb{F}}_p\cong \bar{\mathbb{F}}_p^{\sum_{s\geq1}\alpha_{ps}},
$
and also
$
H_k(\mathbb{B}(\mathbb{Z}))\otimes_\mathbb{Z}\mathbb{Q}\cong \mathbb{Q}^\alpha.
$

Thus, in order to prove the above claim, we only need to show that $H_i(\mathbb{B}(\mathbb{Z}))\otimes_\mathbb{Z}\mathbb{Q}=0$ and
$H_i(\mathbb{B}(\mathbb{Z}))\otimes_\mathbb{Z}\bar{\mathbb{F}}_p=0$
for any prime number $p$.

Let $\mathbb{L}$ denote either $\mathbb{Q}$ or the field $\bar{\mathbb{F}}_p$ for a
prime $p$. Then, $H_i(\mathbb{B}(\mathbb{Z}))\otimes_\mathbb{Z}\mathbb{L}$ is
a submodule of $H_i(\mathbb{B}(\mathbb{L}))$ due to the Universal Coefficient Theorem. Hence, in order to prove $H_i(\mathbb{B}(\mathbb{Z})\otimes_\mathbb{Z}\mathbb{L}) = 0$, it is enough
to show that $H_i(\mathbb{B}(\mathbb{L}))=0$.

Now we  prove that $H_i(\mathbb{B})(\mathbb{L})=0$.

First, the algebra
$S_\mathbb{L}(n,r)$ has an anti-involution $\xi:
S_\mathbb{L}(n,r)\rightarrow S_\mathbb{L}(n,r)$ defined on the basis
elements by $\varphi_{\lambda\mu}^d\mapsto\varphi_{\mu\lambda}^{d^{-1}}$. The
image of $S^+_\mathbb{L}(n,r)$ under $\xi$ is the
subalgebra $S^-_\mathbb{L}(n,r)$ of $S_\mathbb{L}(n,r)$.

As in Definition \ref{involution}, this homomorphism $\xi$ induces a
contravariant equivarience of categories $\mathcal {J}:$
$_{S^+_L(n,r)}\text{\bf mod}\rightarrow$
$_{S^-_L(n,r)}\text{\bf mod}$.

By Lemma \ref{acyclic}, the following exact functors are isomorphic:
\begin{eqnarray*}
\mathcal {J}\circ \text{Hom}_{S^-_L(n,r)}(S_\mathbb{L},-)\circ \mathcal {J}: {}_{S^+_L(n,r)}{\bf mod}\rightarrow {}_{S_L(n,r)}{\bf mod}\\
S_\mathbb{L}(n,r)\otimes_{S^+_L(n,r)}-: {}_{S^+_L(n,r)}{\bf mod}\rightarrow {}_{S_L(n,r)}{\bf mod}.
\end{eqnarray*}

Second, we recall that $\mathbb{B}(\mathbb{L}) = S_\mathbb{L}\otimes_{S^+_L} \mathbb{C}_*^+(L_\lambda)\cong \mathcal{J}(\text{Hom}_{S^-_L}(S_L,\mathcal {J}(B_*(L_\lambda))))$.
Furthermore, by \emph{Kempf's vanishing theorem} we already know that the module
$L_\lambda$ is $\text{Hom}_{S^-_L}(S_\mathbb{L},-)$-acyclic. Thus, after
applying $\mathcal {J}$, we get that the complex $\mathbb{B}(\mathbb{L})$ is exact.
\end{proof}

\section{\bf The Boltje-Maisch complex}

Suppose that $n\geq r$, then we know that there is a partition
$\delta:=(1,\ldots,1,0,\ldots,0)\in\Lambda(n,r)$. Then there is a
obvious isomorphism of algebras $\phi:\mathcal {H}_r\cong
\psi_\delta S_R(n,r)\psi_\delta$. Therefore we can tell that if $M$ is an
$S_R(n,r)$-module, the $\psi_\delta
S_R(n,r)\psi_\delta$-module $\psi_\delta M$ induced by $\phi$ which is a $R\mathcal {H}_r$-module too.
In fact the map $M\mapsto \psi_\delta M$ is functorial, and it defines a functor
$\mathscr{S}_r:{}_{S_R(n,r)}{\bf mod}$$\rightarrow {}_{R\mathcal{H}_r}{\bf mod}$
 which was named \emph{Schur functor} in \cite{23}.

In this section we show that for any $\lambda\in\Lambda^+(n,r)$ the
complex $\mathscr{S}_r(S_R\otimes_{S^+_R}\mathbb{C}_*^+(L_\lambda))$ is
isomorphic  to the complex which has been constructed in \cite{22}. Here, we call it
{\em{Boltje-Maisch complex}}.

First of all, we start with some notations and conventions.
\begin{definition}\cite{24}
For any $\lambda,\mu\in\Lambda(n,r)$,
there is a $R$-submodule of
$Hom_\mathcal{H}(M^\mu,M^\lambda)$. One can denote it by $Hom_{\mathcal{H}_r}^\wedge(M^\mu,M^\lambda)$,
which is a free $R$-module in the following form:
\begin{eqnarray}
\text{Hom}^\wedge_{\mathcal{H}_r}(M^\mu, M^\lambda):=\bigoplus\limits_{d\in \Omega_{\lambda\mu}^{\succeq}}R\psi_{\lambda\mu}^d.
\end{eqnarray}
\end{definition}

Since $\psi_\lambda S^+_R(n,r)\psi_\mu=\psi_{\lambda\lambda}^1\bigoplus\limits_{\substack{d\in \Omega_{\theta,\eta}^{\succeq}\\
\theta,\eta\in\Lambda(n,r)}}R\psi_{\theta\eta}^d\psi_{\mu\mu}^1=\bigoplus\limits_{d\in \Omega_{\lambda\mu}^{\succeq}}R\psi_{\lambda\mu}^d$,
 we know that $\text{Hom}^\wedge_{\mathcal{H}_r}(M^\mu, M^\lambda)$ equals $\psi_\lambda S^+_R(n,r)\psi_\mu$.
 Moreover, for $S^+_R(n,r)=J_1\oplus R_\Lambda$, we have $\text{Hom}^\wedge_{\mathcal{H}_r}(M^\mu,M^\lambda)=\psi_\lambda J_1\psi_\mu$ if $\lambda\rhd\mu$.

Boltje and Maisch defined a complex $\widetilde{B}^\lambda_*$ in Section {3.1} of \cite{21},
as following:

For some $\lambda\in \Lambda^+(n,r)$,
$\widetilde{B}^\lambda_{-1}$ is the dual Specht module that relates to $\lambda$ and $\widetilde{B}^\lambda_0$ defined as $\text{Hom}_R(M^\lambda,R)$. When $k\geq1$,
$\widetilde{B}^\lambda_k$ is defined as the direct sum over all sequence
$(\mu^{(1)}\rhd\cdots\rhd\mu^{(k)}\rhd\lambda)$
 as
\begin{eqnarray*}\label{summ}
\bigoplus\limits_{\substack{\mu^{(1)}\rhd\cdots\rhd\mu^{(k)}\rhd\lambda \\ \mu^{(1)},\ldots,\mu^{(k)}\in \Lambda^+(n,r)} }\text{Hom}_R(M^{\mu^{(1)}},R)\otimes_R \text{Hom}_{\mathcal{H}_r}^\wedge(M^{\mu^{(2)}},M^{\mu^{(1)}})\otimes_R\cdots\otimes_R \text{Hom}_{\mathcal{H}_r}^\wedge(M^\lambda,M^{\mu^{(k)}}).
\end{eqnarray*}

The differential $d_k$, $k\geq1$, in $\widetilde{B}^\lambda_*$ is given by the formula
\begin{eqnarray}
d_k(f_0\otimes f_1\otimes\cdots\otimes f_k)=\sum^{k-1}_{t=0}(-1)^tf_0\otimes\cdots\otimes f_t\circ f_{t+1}\otimes\cdots\otimes f_k.
\end{eqnarray}
and when $k=0$, we put
\begin{eqnarray}
d_0^\lambda:\widetilde{B}_0^\lambda=\text{Hom}_R(M^\lambda,R)\rightarrow \text{Hom}_R(S^\lambda,R)=\widetilde{B}^\lambda_{-1}, \qquad \varepsilon\mapsto \varepsilon|_{S^\lambda}.
\end{eqnarray}

and finally obtain a chain complex with only finite no trivial terms:

\begin{eqnarray}
\widetilde{B}^\lambda_*:\qquad 0\rightarrow
\widetilde{B}^\lambda_{f(\lambda)}\xrightarrow{d^\lambda_{f(\lambda)}}\widetilde{B}^\lambda_{f(\lambda)-1}\xrightarrow{d^\lambda_{f(\lambda)-1}}\cdots\xrightarrow{d^\lambda_1}\widetilde{B}^\lambda_0\xrightarrow{d^\lambda_0}\widetilde{B}^\lambda_{-1}\rightarrow 0,
\end{eqnarray}
where $f(\lambda)$ is a positive integer by Proposition \ref{finitel}.

According to Theorem \textbf{4.2} and \textbf{4.4} in \cite{24}, we have:
\begin{lemma}\cite{21}\label{lemmanew}
  $\widetilde{B}^\lambda_*$ are exact in degree $0$ and $-1$.
\end{lemma}

\begin{theorem}\label{main2}
For $\lambda\in \Lambda^+(n,r)$, the complex $\widetilde{B}_*^\lambda$ is isomorphic to the complex $$\mathscr{S}_r(S_R(n,r)\otimes_{S^+_R(n,r)} \mathbb{C}_*^+(L_\lambda)).$$
\end{theorem}

\begin{proof}
For convenience, we denote the complex $\mathscr{S}_r(S_R(n,r)\otimes_{S^+_R(n,r)} \mathbb{C}_*^+(L_\lambda))$ by $\widehat{B}^\lambda_*$.

 Since the complex $\widehat{B}^\lambda_*$ is
exact, and the exactness of $\widetilde{B}^\lambda_*$ in degree
$0$ and $-1$ has been treated in Lemma \ref{lemmanew}, we only need to establish the isomorphism  in the non-negative
degrees. The isomorphism in the degree $-1$ will follow.

With the consequence of Proposition \ref{simple}, we can write the factor $\widehat{B}^\lambda_k$ as a direct sum of
\begin{eqnarray}\label{note}
\psi_\delta S_R(n,r)\psi_{\mu^{(1)}}\otimes_R \psi_{\mu^{(1)}}J_1\psi_{\mu^{(2)}}\otimes_R \cdots\otimes_R\psi_{\mu^{(k)}}J_1\psi_\lambda,
\end{eqnarray}
where subscripts satisfy $\mu^{(1)}\rhd\cdots\rhd\mu^{(k)}\rhd\lambda$. Furthermore, it is straightforward that the summand (\ref{note}) is isomorphic to
\begin{eqnarray*}
  \text{Hom}_{\mathcal {H}_r}(M^{\mu^{(1)}},\mathcal {H}_r)\otimes_R \text{Hom}_{\mathcal {H}_r}^\wedge(M^{\mu^{(2)}},M^{\mu^{(1)}})\otimes_R\cdots\otimes_R \text{Hom}_{\mathcal {H}_r}^\wedge(M^\lambda,M^{\mu^{(k)}}),
  \end{eqnarray*}

 First, to show the correspondence of factors between $\widehat{B}^\lambda_k$ and $\widetilde{B}^\lambda_*$ in non-negative degrees, it is enough an isomorphism of $\mathcal{H}_r$-modules $\phi_\nu:\text{Hom}_{\mathcal{H}_r}(M^\nu,\mathcal {H}_r)\rightarrow \text{Hom}_R(M^\nu, R)$, to find for every $\nu\in \Lambda(n,r)$.

With the natural anti-automorphism of $\chi:\mathcal{H}_r\rightarrow \mathcal{H}_r$ and $f\in \text{Hom}_{\mathcal{H}_r}(M^\nu, \mathcal{H}_r)$,
we can define right module structure of $\mathcal{H}_r$ on $\text{Hom}_{\mathcal{H}_r}(M^\nu,\mathcal{H}_r)$ and $\text{Hom}_R(M^\mu, R)$, which is given by the formula $(f\sigma)(m)=f(m\cdot\chi(\sigma))$, where $m\in M^\nu$, and $\sigma\in \mathcal{H}_r$.

 For $f\in \text{Hom}_{\mathcal{H}_r}(M^\nu,\mathcal {H}_r)$ and  $m\in M^\nu$,  define $\phi_\nu(f)(m)$ to be the coefficient of $T_{id}$ in $f(m)\in \mathcal{H}_r$. Note that $\{T_w|w\in \mathfrak{S}_r\}$ is an $R$-basis of $\mathcal{H}_r$, now we claim that $\phi_\nu$ is a homomorphism of $\mathcal{H}_r$-modules.

First, we have that  $\phi_\nu(f)\cdot T_w=\phi_\nu(fT_w)$ because:

(i) $(\phi_\nu(f)\cdot T_w)(m)=\phi_\nu(f)(m T_{w^{-1}})=\text{the coefficient}\  \text{of}\  T_{id}\  \text{in}\  f(m T_{w^{-1}})$,

(ii) $\phi_\nu(fT_w)(m)=\text{the coefficient}\  \text{of}\  T_{id}\  \text{in}\  (f\cdot T_w)(m)=\text{the coefficient}\  \text{of}\  T_{id}\  \text{in}\  f(m T_{w^{-1}}).$

Assume that $\phi_\nu(f)=0$, then the coefficient of
$T_{id}$ in $f(m)$ equals to 0, for any $m\in M^\nu$. By induction on
the length of elements of $\mathfrak{S}_r$, suppose that the
coefficient of $T_w$ in $f(m)$ equals to 0, for any $m\in M^\nu$ and $\ell(w)\leq n$.

Let $w'\in \mathfrak{S}_r$ with $\ell(w')=n+1$.
Using the Mathas's formula in \cite{7}, one can find $s=(i, i+1)\in \mathfrak{S}_r$ for some $1\leq i\leq
r-1$ such that
\begin{eqnarray*}
T_{w'}T_s=qT_{w's}+(q-1)T_{w'}, \qquad and  \qquad  n=\ell(w's)<\ell(w')=n+1.
\end{eqnarray*}
Meantime, $f(m)T_s=f(mT_s)$,  and write $f(m)=\sum\limits_{\ell(w)>n}\alpha_w T_w$ for some $\alpha_w\in R$. Then we have $$0=f(mT_s)=\sum\limits_{\ell(w)>n}\alpha_w T_w \cdot T_s=q\alpha_{w} T_{ws}+\sum\limits_{\substack{\ell(w)\geq n\\w\neq w_0s}}\alpha'_w T_w.$$
It follows that $q\alpha_{w}=0$ then $\alpha_{w} = 0$ since $q\not=0$. Therefore, we find out that $f(m)=0$ for any $m\in M^\nu$ and then $f=0$,
which implies the homomorphism $\phi_\nu$ is injective.

As free $R$-modules, it is obvious that the rank of $\text{Hom}_{\mathcal{H}_r}(M^\nu,\mathcal{H}_r)$ equals to $\text{Hom}_R(M^\nu, R)$, which means the homomorphism $\phi_\nu$ is surjective. It implies that $\phi_\nu$ is an isomorphism.

Lastly, by Definition \ref{bar}, it is straightforward that the isomorphism $\phi_\nu$ indeed induces a chain map from
$\widetilde{B}^\lambda_*$ onto $\widehat{B}^\lambda_*$, which is just the isomorphism from $\widetilde{B}^\lambda_*$ to
$\widehat{B}^\lambda_*$.
\end{proof}

\begin{remark}
With help of the above theorem and Theorem \ref{main1}, we have shown that the Boltje-Maisch complex $\widetilde{B}^\lambda_*$ is exact for $\lambda\in \Lambda^+(n,r)$.
\end{remark}

\noindent
{\bf Acknowledgements: }{\em The authors thank the support from the projects of the National Natural Science Foundation of China (No.11271318 and No.11171296) and the Specialized Research Fund for the Doctoral Program of Higher Education of China (No.20110101110010).  }

\end{document}